\documentclass{amsart}
\usepackage{amsmath}
\usepackage{amsfonts}
\newcommand{\N}{\mathbb{N}}
\newcommand{\R}{\mathbb{R}}

\newtheorem{theorem}{Theorem}
\theoremstyle{plain}

\newtheorem{definition}{Definition}

\numberwithin{equation}{section}
\input{tcilatex}

\begin{document}
\title[Coupled fixed point theorems in partially ordered metric spaces]{Coupled fixed point theorems for contraction involving  rational expressions in partially ordered metric spaces}
\author{Bessem Samet}
\address{DEPARTMENT OF MATHEMATICS, TUNIS COLLEGE OF SCIENCES AND TECHNIQUES, 5 AVENUE TAHA HUSSEIN, BP, 59, BAB MANARA, TUNIS. }
\email{bessem.samet@gmail.com}
\author{Habib Yazidi}
\address{DEPARTMENT OF MATHEMATICS, TUNIS COLLEGE OF SCIENCES AND TECHNIQUES, 5 AVENUE TAHA HUSSEIN, BP, 59, BAB MANARA, TUNIS.}
\email{Habib.Yazidi@esstt.rnu.tn}
\begin{abstract}
We establish coupled fixed point theorems for contraction involving  rational expressions in partially ordered metric spaces. 
\end{abstract}
\maketitle

\section{Introduction and preliminaries}
A number of generalizations of the well-known Banach contraction
theorem \cite{BAN} were obtained in various directions.   Most of these deal with the generalizations of the contractive condition (see \cite{beg, DG, de, hard, J, J2, J3, J4, J5, Matk, MK, p, SR, Rh, sessa} and others). 
Recently, Bhaskar and Lakshmikantham \cite{BL}, Nieto and Lopez \cite{N, N2}, Ran and Reurings \cite{R} and Agarwal, El-Gebeily and O'Regan \cite{A} presented some new results for contractions in partially ordered metric spaces.

In \cite{BL}, Bhaskar and Lakshmikantham introduced the notion of a coupled fixed point 
and proved some coupled fixed point theorems for mappings satisfying a mixed monotone
property. They discussed the problems of a uniqueness of a coupled fixed point and
applied their theorems to problems of the existence and uniqueness of solution for a periodic boundary value problem.  Lakshmikantham and \'Ciri\'c \cite{L} introduced
the concept of a mixed g-monotone mapping and proved coupled coincidence
and coupled common fixed point theorems that extend theorems due to Bhaskar
and Lakshmikantham \cite{BL}.

In this paper, we derive new coupled fixed point theorems in partially ordered metric space for mappings satisfying a mixed monotone property and a contractive condition involving  rational expression. Our obtained results generalize theorems due to Bhaskar
and Lakshmikantham \cite{BL}.

Before presenting the main results of the paper, we start by recalling some definitions introduced in \cite{BL}.

\begin{definition}\label{def1}
Let $(X,\leq)$ be a partially ordered set and $F: X\times X\rightarrow X$. We say that
$F$ has the mixed monotone property if for any $x, y\in X$, 
\begin{eqnarray*}
x_1, x_2\in X, \quad x_1\leq x_2 \Rightarrow F(x_1,y)\leq F(x_2,y),\\ 
y_1, y_2\in X, \quad y_1\leq y_2 \Rightarrow F(x,y_1)\geq F(x,y_2).
\end{eqnarray*}
\end{definition}
\noindent This definition coincides with the notion of a mixed monotone function on $\R^2$ and
 $\leq $ represents the usual total order in $\R$.
\begin{definition}\label{def2}
We call an element $(x,y)\in X \times X$ a coupled fixed point of the mapping 
$F: X\times X \rightarrow X$ if
$$
F(x,y)=x \quad \mbox{ and } \quad F(y,x)=y.
$$
\end{definition}
Let $(X,\leq)$ be a partially ordered set. Further, we endow the product space $X\times X$ with the following partial order:
$$
\mbox{ for } (x,y), (u,v) \in X \times X,\quad  (u,v)\leq (x,y)\Leftrightarrow x\geq u, \,\, y \leq v.
$$
\section{Coupled fixed point theorems}
\subsection{Existence of a coupled fixed point}

We start by the following result.
\begin{theorem}\label{T1}
Let $(X,\leq)$ be a partially ordered set and $d$ be a metric on $X$ such that 
$(X,d)$ is a complete metric space. Let $F: X \times X \rightarrow X$ be a continuous mapping having the mixed monotone property on $X$. For all $(x,y), (u,v) \in X\times X$, we denote by $M((x,y),(u,v))$ the quantity: 
$$
\min\left\{d(x,F(x,y))\,\frac{2+d(u,F(u,v))+d(v,F(v,u))}{2+d(x,u)+d(y,v)}, d(u,F(u,v))\,\frac{2+d(x,F(x,y))+d(y,F(y,x))}{2+d(x,u)+d(y,v)}\right\}.
$$
Assume that there exists
$\alpha, \beta> 0$ with $\alpha+\beta<1$ such that
\begin{equation}\label{contraction}
d(F(x,y),F(u,v))\leq \alpha M((x,y),(u,v))+\frac{\beta}{2}[d(x,u)+d(y,v)]
\end{equation}
for all $(x,y), (u,v)\in X\times X$ with $x\geq u$ and $y\leq v$.
We assume that there exists $x_0, y_0\in X$ such that
\begin{equation}\label{cond}
x_0\leq F(x_0,y_0) \quad \mbox{ and } \quad y_0\geq F(y_0,x_0).
\end{equation}
Then, $F$ has a coupled fixed point $(x,y)\in X\times X$.
\end{theorem}
\begin{proof}
Since $x_0\leq F(x_0,y_0)=x_1$ (say) and $y_0\geq F(y_0,x_0)=y_1$ (say), letting
$x_2=F(x_1,y_1)$ and $y_2=F(y_1,x_1)$. We denote:
$$
F^2(x_0,y_0)=F(F(x_0,y_0),F(y_0,x_0))=F(x_1,y_1)=x_2
$$
and
$$
F^2(y_0,x_0)=F(F(y_0,x_0),F(x_0,y_0))=F(y_1,x_1)=y_2.
$$
Due to the mixed property of $F$, we have:
$$
x_2=F(x_1,y_1) \geq F(x_0,y_1)\geq F(x_0,y_0)=x_1
$$
and
$$
y_2=F(y_1,x_1)\leq F(y_0,x_1)\leq F(y_0,x_0).
$$
Further, for $n= 1, 2, \cdots $, we let,
$$
x_{n+1}=F^{n+1}(x_0,y_0)=F(F^n(x_0,y_0),F^n(y_0,x_0))
$$
and
$$
y_{n+1}=F^{n+1}(y_0,x_0)=F(F^n(y_0,x_0),F^n(x_0,y_0)).
$$
We check easily that
$$
x_0\leq x_1\leq x_2\leq \cdots \leq x_{n+1}\leq \cdots
$$
and
$$
y_0\geq y_1 \geq y_2\geq \cdots \geq y_{n+1}\geq \cdots
$$
Now, we claim that, for $n\in\N^*$,
\begin{equation}\label{b1}
d(x_{n+1},x_n)\leq \left(\frac{\beta}{1-\alpha}\right)^n\frac{[d(x_1,x_0)+d(y_1,y_0)]}{2}
\end{equation}
and
\begin{equation}\label{b2}
d(y_{n+1},y_n)\leq \left(\frac{\beta}{1-\alpha}\right)^n\frac{[d(x_1,x_0)+d(y_1,y_0)]}{2}\cdot
\end{equation}
Indeed, for $n=1$, using $x_1\geq x_0$, $y_1\leq y_0$ and (\ref{contraction}), we get:
\begin{eqnarray*}
d(x_2,x_1)&=&d(F(x_1,y_1),F(x_0,y_0))\\
&\leq & \alpha M((x_1,y_1),(x_0,y_0))+\beta \frac{[d(x_0,x_1)+d(y_0,y_1)]}{2}\\
&\leq & \alpha d(x_1,F(x_1,y_1))\,\frac{2+d(x_0,F(x_0,y_0))+d(y_0,F(y_0,x_0))}{2+d(x_0,x_1)+d(y_0,y_1)}\\
&&+\beta\frac{[d(x_0,x_1)+d(y_0,y_1)]}{2}\\
&=&\alpha d(x_1,x_2)+\beta\frac{[d(x_0,x_1)+d(y_0,y_1)]}{2}\cdot
\end{eqnarray*}
This implies that
$$
(1-\alpha)d(x_2,x_1)\leq \beta\frac{[d(x_0,x_1)+d(y_0,y_1)]}{2},
$$
i.e.,
$$
d(x_2,x_1)\leq \left(\frac{\beta}{1-\alpha}\right)\frac{[d(x_0,x_1)+d(y_0,y_1)]}{2}.
$$
Similarly, 
\begin{eqnarray*}
d(y_2,y_1)&=&d(F(y_1,x_1),F(y_0,x_0))\\
&=&d(F(y_0,x_0),F(y_1,x_1))\\
&\leq & \alpha M((y_0,x_0),(y_1,x_1))+\beta \frac{[d(x_0,x_1)+d(y_0,y_1)]}{2}\\
&\leq & \alpha d(y_1,F(y_1,x_1))\,\frac{2+d(y_0,F(y_0,x_0))+d(x_0,F(x_0,y_0))}{2+d(y_0,y_1)+d(x_0,x_1)}\\
&&+\beta \frac{[d(x_0,x_1)+d(y_0,y_1)]}{2}\\
&=& \alpha d(y_1,y_2)+\beta \frac{[d(x_0,x_1)+d(y_0,y_1)]}{2},
\end{eqnarray*}
which implies that
$$
d(y_2,y_1)\leq \left(\frac{\beta}{1-\alpha}\right)\frac{[d(x_0,x_1)+d(y_0,y_1)]}{2}.
$$
Now, assume that (\ref{b1}) and (\ref{b2}) hold. Using $x_{n+1}\geq x_n$ and $y_{n+1}\leq y_n$, we get
\begin{eqnarray*}
d(x_{n+2},x_{n+1})&=&d(F(x_{n+1},y_{n+1}),F(x_n,y_n))\\
&\leq & \alpha M((x_{n+1},y_{n+1}),(x_n,y_n))+\beta \frac{[d(x_{n+1},x_n)+d(y_{n+1},y_n)]}{2}\\
&\leq & \alpha d(x_{n+1},x_{n+2})\,\frac{2+d(x_n,x_{n+1})+d(y_n,y_{n+1})}{2+d(x_{n+1},x_n)+d(y_{n+1},y_n)}\\
&&+\beta \frac{[d(x_{n+1},x_n)+d(y_{n+1},y_n)]}{2}\\
&=&\alpha d(x_{n+1},x_{n+2})+\beta \frac{[d(x_{n+1},x_n)+d(y_{n+1},y_n)]}{2}.
\end{eqnarray*}
This implies that
\begin{eqnarray*}
d(x_{n+2},x_{n+1})&\leq & \left(\frac{\beta}{1-\alpha}\right) \frac{[d(x_{n+1},x_n)+d(y_{n+1},y_n)]}{2}\\
&\leq & \left(\frac{\beta}{1-\alpha}\right)^{n+1}\frac{[d(x_0,x_1)+d(y_0,y_1)]}{2}\cdot
\end{eqnarray*}
Similarly, one can show that
$$
d(y_{n+2},y_{n+1})\leq \left(\frac{\beta}{1-\alpha}\right)^{n+1}\frac{[d(x_0,x_1)+d(y_0,y_1)]}{2}\cdot
$$
Since $0<\displaystyle \frac{\beta}{1-\alpha}<1$, (\ref{b1}) and (\ref{b2})
imply that $\{x_n\}$ and $\{y_n\}$ are Cauchy sequences in $X$.

Since $(X,d)$ is a complete metric space, there exists $(x,y)\in X\times X$ such that
\begin{equation}\label{fc}
\lim_{n\rightarrow +\infty} x_n=x \quad \mbox{ and } \quad \lim_{n\rightarrow +\infty} y_n=y.
\end{equation}
Finally, we claim that $(x,y)$ is a coupled fixed point of $F$. 
In fact, we have:
$$
d(F(x,y),x)\leq d(F(x,y),x_{n+1})+d(x_{n+1},x)=d(F(x,y),F(x_n,y_n))+d(x_{n+1},x).
$$
From (\ref{fc}) and the continuity of $F$, we get immediately that
$$
\lim_{n\rightarrow +\infty} d(F(x,y),F(x_n,y_n))+d(x_{n+1},x)=0.
$$
Then, $F(x,y)=x$. Similarly,
$$
d(F(y,x),y)\leq d(F(y,x),y_{n+1})+d(y_{n+1},y)=d(F(y,x),F(y_n,x_n))+d(y_{n+1},y)
$$
and
$$
\lim_{n\rightarrow +\infty} d(F(y,x),F(y_n,x_n))+d(y_{n+1},y)=0,
$$
which implies that $F(y,x)=y$. This makes end to the proof.
\end{proof}

As in \cite{BL}, Theorem \ref{T1} is also valid if we replace the continuity hypothesis
of $F$ by an additional  property satisfied by the metric space $X$.
This is the aim of the following theorem.
\begin{theorem}\label{T2}
Let $(X,\leq)$ be a partially ordered set and $d$ be a metric on $X$ such that 
$(X,d)$ is a complete metric space. Assume that $X$ has the following property:
\begin{enumerate}
\item[(i)] if a nondecreasing sequence $\{x_n\}$ in $X$  converges to $x\in X$, then
$x_n\leq x$ for all $n$,
\item [(ii)] if a nonincreasing sequence $\{y_n\}$ in $X$ converges to $y\in X$, then
$y_n\geq y$ for all $n$.
\end{enumerate}
Let $F: X \times X \rightarrow X$ be a mapping having the mixed monotone property on $X$. Assume that there exists $\alpha, \beta> 0$ with $\alpha+\beta<1$ such that
$$
d(F(x,y),F(u,v))\leq \alpha M((x,y),(u,v))+\frac{\beta}{2}[d(x,u)+d(y,v)]
$$
for all $(x,y), (u,v)\in X\times X$ with $x\geq u$ and $y\leq v$.
We assume that there exists $x_0, y_0\in X$ such that
$$
x_0\leq F(x_0,y_0) \quad \mbox{ and } \quad y_0\geq F(y_0,x_0).
$$
Then, $F$ has a coupled fixed point $(x,y)\in X\times X$.
\end{theorem}
\begin{proof}
Following the proof of Theorem \ref{T1}, we only have to show that $(x,y)$ is a coupled fixed point of $F$. We have:
\begin{equation}\label{bhama}
d(F(x,y),x) \leq  d(F(x,y),x_{n+1})+d(x_{n+1},x)=d(F(x,y),F(x_n,y_n))+d(x_{n+1},x).
\end{equation}
Since the nondecreasing sequence $\{x_n\}$ converges to $x$ and the nonincreasing sequence $\{y_n\}$ converges to $y$, by (i)-(ii), we have:
$$
x\geq x_n \quad \mbox{ and } \quad y\leq y_n,\quad \forall\,n.
$$
Now, from the contraction condition, we have:
\begin{eqnarray*}
d(F(x,y),F(x_n,y_n))&\leq &\alpha M((x,y),(x_n,y_n))+\frac{\beta}{2}[d(x,x_n)+d(y,y_n)]\\
&\leq & \alpha d(x,F(x,y))\,\frac{2+d(x_n,x_{n+1})+d(y_n,y_{n+1})}{2+d(x,x_n)+d(y,y_n)}\\
&&+\frac{\beta}{2}[d(x,x_n)+d(y,y_n)].
\end{eqnarray*}
Then, from (\ref{bhama}), we get:
\begin{eqnarray*}
d(F(x,y),x) &\leq & \alpha d(x,F(x,y))\,\frac{2+d(x_n,x_{n+1})+d(y_n,y_{n+1})}{2+d(x,x_n)+d(y,y_n)}\\
&&+\frac{\beta}{2}[d(x,x_n)+d(y,y_n)]+d(x_{n+1},x)\\
&\rightarrow & \alpha d(x,F(x,y)) \mbox{ as } n\rightarrow +\infty.
\end{eqnarray*}
This implies that
$$
(1-\alpha)d(F(x,y),x)\leq 0.
$$
Since $0<\alpha<1$, we obtain $d(F(x,y),x)=0$, i.e., $F(x,y)=x$.

Similarly,
\begin{equation}\label{bham}
d(F(y,x),y)\leq d(F(y,x),y_{n+1})+d(y_{n+1},y)=d(F(y_n,x_n),F(y,x))+d(y_{n+1},y).
\end{equation}
From the contraction condition, we get:
$$
d(F(y_n,x_n),F(y,x))\leq \alpha d(y,F(y,x))\,\frac{2+d(y_n,y_{n+1})+d(x_n,x_{n+1})}{2+d(y_n,y)+d(x_n,x)}+\frac{\beta}{2}[d(y_n,y)+d(x_n,x)].
$$

Then, from (\ref{bham}), we get:
\begin{eqnarray*}
d(F(y,x),y)&\leq & \alpha d(y,F(y,x))\,\frac{2+d(y_n,y_{n+1})+d(x_n,x_{n+1})}{2+d(y_n,y)+d(x_n,x)}+\frac{\beta}{2}[d(y_n,y)+d(x_n,x)]\\
&&+d(y_{n+1},y)\\
&\rightarrow & \alpha d(y,F(y,x)) \mbox{ as } n\rightarrow +\infty.
\end{eqnarray*}
This implies that
$$
(1-\alpha)d(F(y,x),y)\leq 0.
$$
Hence, $d(F(y,x),y)=0$, i.e., $F(y,x)=y$.
This makes end to the proof. 
\end{proof}

\subsection{Uniqueness of the coupled fixed point}
\begin{theorem}\label{T3}
Assume that  
\begin{equation}\label{H}
\forall (x,y), (x^*,y^*) \in X\times X, \, \exists\, (z_1,z_2)\in X\times X \mbox{  that is comparable to } (x,y) \mbox{ and } (x^*, y^*).
\end{equation}
Adding (\ref{H}) to the hypotheses of Theorem \ref{T1}, we obtain the uniqueness of the coupled fixed point of $F$.
\end{theorem} 
\begin{proof}
Suppose that $(x^*,y^*)$ is another coupled fixed point of $F$, i.e., 
$$
F(x^*,y^*)=x^* \quad \mbox{ and } \quad F(y^*,x^*)=y^*.
$$
Let us prove that
\begin{equation}\label{goal}
d(x,x^*)+d(y,y^*)=0,
\end{equation}
where 
$$
\lim_{n\rightarrow +\infty} F^n(x_0,y_0)=x \quad \mbox{ and } \quad 
\lim_{n\rightarrow +\infty}F^n(y_0,x_0)=y.
$$
As in \cite{BL}, we distinguish  two cases.

First case: $(x,y)$ is comparable to $(x^*,y^*)$ with respect the ordering in $X\times X$.\\
In this case,  for all $n\in \N$, 
$$
(F^n(x,y),F^n(y,x))=(x,y) \mbox{ is comparable to }(F^n(x^*,y^*),F^n(y^*,x^*))=(x^*,y^*).
$$
Also,
\begin{eqnarray*}
d(x,x^*)+d(y,y^*)&=&d(F^n(x,y),F^n(x^*,y^*))+d(F^n(y,x),F^n(y^*,x^*))\\
&\leq & \beta^n[d(x,x^*)+d(y,y^*)].
\end{eqnarray*}
Since $0<\beta<1$, (\ref{goal}) holds.

Second case: $(x,y)$ is not comparable to $(x^*,y^*)$.\\ In this case, there exists
$(z_1,z_2)\in X\times X$ that is comparable to $(x,y)$ and $(x^*,y^*)$.
Then, for all $n\in \N$, $(F^n(z_1,z_2),F^n(z_2,z_1))$ is comparable to $(F^n(x,y), F^n(y,x))=(x,y)$ and $(F^n(x^*,y^*),F^n(y^*,x^*))=(x^*,y^*)$. We have:
\begin{eqnarray*}
d(x,x^*)+d(y,y^*)&=&d(F^n(x,y),F^n(x^*,y^*))+d(F^n(y,x),F^n(y^*,x^*))\\
&\leq & d(F^n(x,y), F^n(z_1,z_2))+d(F^n(z_1,z_2),F^n(x^*,y^*))\\
&&+d(F^n(y,x),F^n(z_2,z_1))+d(F^n(z_2,z_1),F^n(y^*,x^*))\\
&\leq & \beta^n[d(x,z_1)+d(y,z_2)+d(x^*,z_1)+d(y^*,z_2)].
\end{eqnarray*}
Since $0<\beta<1$, we have:
$$
\lim_{n\rightarrow +\infty}\beta^n[d(x,z_1)+d(y,z_2)+d(x^*,z_1)+d(y^*,z_2)]=0,
$$
which implies that (\ref{goal}) holds.

We deduce that in all cases (\ref{goal}) holds. This implies that $(x,y)=(x^*,y^*)$ and the uniqueness of the coupled fixed point of $F$ is proved.
\end{proof}

\subsection{Equality between the components of the coupled fixed point}

If $x_0, y_0$ in $X$ are comparable, we have the following result.
\begin{theorem}\label{france}
In addition to the hypotheses of Theorem \ref{T1} (resp. Theorem \ref{T2}), suppose that $x_0, y_0$ in $X$ are comparable. Then $x=y$.
\end{theorem}
\begin{proof}
Suppose that $x_0\leq y_0$. We claim that
\begin{equation}\label{hG}
x_n\leq y_n,\,\,\forall\,n\in \N.
\end{equation}
From the mixed monotone property of $F$, we have:
$$
x_1=F(x_0,y_0)\leq F(y_0,y_0)\leq F(y_0,x_0)=y_1.
$$
Assume that $x_n\leq y_n$ for some $n$. Now, 
\begin{eqnarray*}
x_{n+1}&=&F^{n+1}(x_0,y_0)\\
&=&F(F^n(x_0,y_0),F^n(y_0,x_0))\\
&=&F(x_n,y_n)\\
&\leq & F(y_n,y_n)\\
&\leq & F(y_n,x_n)=y_{n+1}.
\end{eqnarray*}
Hence, (\ref{hG}) holds. 

Now, using (\ref{hG}) and the contraction condition, we get:
\begin{eqnarray*}
d(x,y)&\leq & d(x, x_{n+1})+d(x_{n+1},y_{n+1})+d(y_{n+1},y)\\
&=& d(x,x_{n+1})+d(F(F^n(x_0,y_0),F^n(y_0,x_0)),F(F^n(y_0,x_0),F^n(x_0,y_0)))+d(y_{n+1},y)\\
&=&d(x,x_{n+1})+d(F(y_n,x_n),F(x_n,y_n))+d(y_{n+1},y)\\
&\leq & d(x,x_{n+1})+\alpha M((y_n,x_n),(x_n,y_n))+\beta d(x_n,y_n)+d(y_{n+1},y)\\
&\leq & d(x,x_{n+1})+\alpha d(y_n,y_{n+1})\,\frac{2+d(x_n,x_{n+1})+d(y_n,y_{n+1})}{2+2d(y_n,x_n)}\\&&+\beta d(x_n,y_n)+d(y_{n+1},y)\\
&\leq & d(x,x_{n+1})+\alpha d(y_n,y_{n+1})\,\frac{2+d(x_n,x_{n+1})+d(y_n,y_{n+1})}{2}\\&&+\beta [d(x_n,x)+d(y,y_n))]+d(y_{n+1},y)+\beta d(x,y).
\end{eqnarray*}
Hence,
\begin{eqnarray*}
(1-\beta)d(x,y)&\leq & d(x,x_{n+1})+\alpha d(y_n,y_{n+1})\,\frac{2+d(x_n,x_{n+1})+d(y_n,y_{n+1})}{2}\\&&+\beta [d(x_n,x)+d(y,y_n))]+d(y_{n+1},y)\\
&\rightarrow & 0 \mbox{ as } n\rightarrow +\infty.
\end{eqnarray*}
Since $0<\beta<1$, this implies that $d(x,y)=0$, i.e., $x=y$.
This makes end to the proof.
\end{proof}

\subsection{Remark} 
If we put:
$$
f(x)=F(x,x),\,\,\forall\,x\in X,
$$
for $x=y=\hat{x}$ and $u=v=\hat{y}$, the contraction (\ref{contraction}) implies:
$$
d(f(\hat{x}),f(\hat{y}))\leq \alpha d(\hat{y},f(\hat{y}))\,\frac{1+d(\hat{x},f(\hat{x}))}{1+d(\hat{x},\hat{y})}+\beta d(\hat{x},\hat{y}),
$$
and we retrieve the contraction of Dass and Gupta \cite{DG}.

\end{document}